\documentclass{amsart}
\usepackage{amssymb,mathrsfs}
\usepackage{color,umoline}
\usepackage[dvipsnames]{xcolor}
\usepackage{graphicx}
\usepackage{enumitem}
\usepackage[font=small,labelfont=bf]{caption}
\usepackage{tikz, caption}
\usetikzlibrary{patterns}
\usepackage{relsize}
\usepackage[mathscr]{eucal}
\usepackage{dsfont}
\usepackage{verbatim}
\usepackage[draft]{todonotes}
\usepackage[
citecolor=black,  
hypertexnames=false]{hyperref}
\usepackage[dvipsnames]{xcolor}
\usepackage{subfig}

\newcommand{\C}{\mathbb{C}}
\newcommand{\N}{\mathbb{N}}
\newcommand{\R}{\mathbb{R}}
\newcommand{\Z}{\mathbb{Z}}
\newcommand{\T}{\mathbb{T}}

\newcommand{\re}{\mathop{\mathrm{Re}}}
\newcommand{\im}{\mathop{\mathrm{Im}}}
\newcommand{\supp}{\operatorname{supp}}
\newcommand{\dist}{\operatorname{dist}}

\newcommand{\sgn}{\operatorname{sgn}}

\newtheorem{thm}{Theorem}[section]
\newtheorem{defn}[thm]{Definition}
\newtheorem{prop}[thm]{Proposition}

\newtheorem{lemma}[thm]{Lemma}
\numberwithin{equation}{section}

\newcommand\CI{{\mathcal I}}
\newcommand\CJ{{\mathcal J}}

\newcommand\CL{{\mathcal L}}

\newcommand{\wpsi}{{\widetilde \psi}}

\newcommand{\ol}[1]{\overline{#1}}

\newcommand{\pq}{(\tfrac1p, \tfrac1q)}
\newcommand{\ppq}{(1/p,1/q)}
\newcommand{\Be}{\begin{equation}}
\newcommand{\Ee}{\end{equation}}

\newcommand{\pj}{\mathcal P_{\!\mu}}
\begin{document}

\author[Jeong]{Eunhee Jeong}
\address[Jeong]{School of Mathematics, Korea Institute for Advanced Study, Seoul 02455, Republic of Korea} 
\email{eunhee@kias.re.kr}

\author[Lee]{Sanghyuk Lee}
\author[Ryu]{Jaehyeon Ryu}
\address[Lee, Ryu]{Department of Mathematical Sciences and RIM, Seoul National University, Seoul 08826, Republic of  Korea}
\email{shklee@snu.ac.kr}
\email{miro21670@snu.ac.kr}

\title[$L^p$--$L^q$ spectral projection estimate for the twisted Laplacian]{Sharp $L^p$--$L^q$ estimate for the spectral projection 
 associated  with  the twisted Laplacian}

\begin{abstract} In this note we are concerned with  estimates for the spectral projection operator $\pj$ associated with the twisted Laplacian $L$. 
 We  completely characterize the optimal bounds on the operator norm of $\pj$ from $L^p$ to $L^q$ when $1\le p\le 2\le q\le \infty$.  As an application, we obtain uniform resolvent estimate for $L$. 
\end{abstract}

\maketitle

\section{Introduction}

We consider  the twisted Laplacian  $L$  on $\C^d\cong \R^{2d}$, $d\ge1$, which is defined by 
\[ L = - \sum_{j=1}^d \Big(\big(\frac\partial{\partial x_j}-\frac12iy_j\big)^2 +\big(\frac\partial{\partial{y_j}}+\frac12 ix_j\big)^2\Big).\]
It is well known that $L$ has a discrete spectrum which  consists of the points $2k+d$, $k\in \N_0  :=\N\cup\{0\}.$ For any multi-index $\alpha, \beta\in \N_0^d,$ the special Hermite function $\Phi_{\alpha, \beta}$ is given  by
\[ \Phi_{\alpha, \beta} (z) =(2\pi )^{-d/2} \int_{\R^d}e^{ix\cdot \xi} \Phi_\alpha(\xi+\frac12y)\Phi_\beta(\xi-\frac12y) d\xi, \quad z=x+iy,\]
which are the Fourier-Wigner transform of the Hermite functions $\Phi_\alpha$ and $\Phi_\beta$ on $\R^d.$ It is easy to see that $\{\Phi_{\alpha, \beta}: \alpha, \beta\in \N_0^d\}$ forms a complete orthonormal system in $L^2(\C^d)$. Also,  $L\Phi_{\alpha,\beta} =(2|\beta|+d)\Phi_{\alpha,\beta}$, which means $\Phi_{\alpha,\beta}$ is an eigenfunction of $L$ with the eigenvalue $2|\beta|+d$, hence the eigenspace of $L$ is infinite dimensional.  Here $|\beta|=\beta_1+\cdots+\beta_d.$ A simple calculation shows that $\Phi_{\alpha,\beta} $ is also an eigenfunction of the Hermite operator $-\Delta_z +\frac14|z|^2$ with the eigenvalue $|\alpha|+|\beta|+d.$ So, the functions $\Phi_{\alpha,\beta}$ are called the special Hermite functions. For more about the twisted Laplacian and the special Hermite functions, we refer the reader to the monograph  by Thangavelu \cite{Th93}.

The spectral projection operator  $\pj$ onto  the eigenspace of $L$ associated  with the eigenvalue ${\mu}=2k+d \in 2\N_0+d$ is given by 
\Be\label{basic}
\pj f= \sum_{\alpha\in \N^d_0}\sum_{\beta : 2|\beta|+d={\mu}} \langle f,\Phi_{\alpha, \beta} \rangle \,\Phi_{\alpha, \beta}, \quad f\in \mathcal S(\R^{2d}).
\Ee
Thus, it follows $f=\sum_{\mu\in 2\N_0+d} \pj f$. 
It is known \cite{Th93} that $\pj$ is also expressed by the twisted convolution: 
\begin{equation}\label{pro_tw} 
\pj f= (2\pi)^{-d} f\times \varsigma_k,  \quad {\mu}=2k+d,
\end{equation}
where  $\varsigma_k(z)= L^{d-1}_k (\frac12|z|^2)e^{-\frac14|z|^2}$ and $L^{\alpha}_k(t) = \sum_{j=0}^k {k+\alpha \choose k-j }\frac{(-t)^j}{j!}$  is the Laguerre polynomial of type $\alpha$. Here, the twisted convolution $f\times g$ is  defined by
\[ f\times g(z) =\int_{\C^d} f(z-w)g(w)e^{i\frac12\im z\cdot \overline w} dw\] 
where  $z\cdot w = z_1w_1+\cdots +z_dw_d$ for any $z,w\in \C^d$.

The estimates for $\pj$ have been of interest  related to  $L^p$ convergence of  the Bochner-Riesz means $S_R^\alpha(L)$ associated with the special Hermite expansion which is given by $S_R^\alpha(L) f:=\sum_{\mu\le R} (1- \mu/R)\pj f$ (see, for example,  \cite{Th93}). 
In particular,  $L^2$--$L^q$ estimate for $\pj$ (equivalently, $L^{q'}$--$L^2$ estimate for $\pj$) was 
studied by Thangavelu \cite{Th91, Th93, Th98},   Ratnakumar, Rawat, and Thangavelu \cite{RRT97},  Stempak and  Zienkiewicz \cite{SZ}, 
and   Koch and Ricci \cite{KR}. The sharp $L^2$--$L^q$ bound for $\pj$ is now well-understood. More precisely,  for $2\le q\le \infty$
\begin{equation}\label{known}
 \|\pj\|_{2\to q}\lesssim {\mu}^{\varrho(q)}
 \end{equation}
holds with the exponent $\varrho(q)$ given by
\[
\varrho(q) =\begin{cases}
	-\frac12(\frac12-\frac1q) &\text{if } 2\le q\le \frac{2(2d+1)}{2d-1},\\
	\frac{d-1}2-\frac dq &\text{if }\frac{2(2d+1)}{2d-1}\le q\le \infty,
	\end{cases}
\]
and the estimate \eqref{known} is optimal in that the exponent $\varrho(q)$ cannot be taken to be a smaller one. Here, $\|T\|_{p\to q}$ denotes the usual operator norm from $L^p$ to $L^q$ of a linear operator $T$ defined by 
\[\|T\|_{p\to q} =\sup_{f\in \mathcal S, f\neq 0 } \|Tf\|_q/\|f\|_p.\]
The estimate \eqref{known} was shown by Thangavelu \cite{Th91, Th93} and, subsequently, Ratnakumar, Rawat, and Thangavelu \cite{RRT97}  for $q\ge q_\circ$ with some $q_\circ>2(2d+1)/(2d-1)$. Afterward, Stempak and Zienkiewicz \cite{SZ} proved \eqref{known} for all $q\ge2$ except $q=2(2d+1)/(2d-1)$. The remaining  end point case $q=2(2d+1)/(2d-1)$ was settled  by Koch and Ricci \cite{KR}; moreover,  they showed  the estimate \eqref{known} is sharp.   A local version of the endpoint estimate was obtained earlier by Thangavelu \cite{Th98}.

The purpose of this paper is to establish the optimal $L^p$--$L^q$ estimate for $\pj$ when $1\le p\le 2\le q\le \infty$. Our result  was inspired by the recent work \cite{JLR} of the authors regarding $L^p$--$L^q$ estimates for the Hermite spectral projection $\Pi_{\mu}$ which is the orthogonal projection onto the eigenspace of the Hermite operator $H=|x|^2-\Delta$ associated with the eigenvalue $\mu$. In \cite{JLR}, a systematic study concerning the bound on $\|\Pi_\mu\|_{p\to q}$ was carried out. The main ingredients were a representation formula for $\Pi_{\mu}$ and a modification of $TT^*$ argument.  In particular,   the representation formula   was obtained by making  use of the Schr\"odinger propagator $e^{itH}$ and the fact that the eigenvalues of $H$ are in $2\N_0+d$. See \cite[Section 2]{JLR}. 
It turns out that the similar approach works even more efficiently for the projection operator $\pj$ and 
we obtain a complete characterization of the  $L^p$--$L^q$ bound on $\|\pj\|_{p\to q}$ in terms of $\mu$.

In Theorem \ref{main} below, we show that the boundedness property of the spectral projection $\pj$ is similar to  that of  the
operator 
\[
{\mathlarger \wp}_k f(x,y)= \frac1{(2\pi)^{2d}}\int_{k-1\le |\xi|^2<k}e^{i(x,y)\cdot\xi}\,\widehat f(\xi)\, d\xi, \quad (x,y)\in \mathbb R^d\times  \mathbb R^d, 
\]
which is the spectral projection operator  associated with the Laplacian  $-\Delta$ in $\R^{2d}$. For a discussion regarding the sharp $L^p$--$L^q$ bounds for the operator ${\mathlarger \wp}_k$, we refer to \cite[Section 3.3]{JLR}. Compared with the Hermite spectral projection $\Pi_\mu$,  the sharp exponent  $\varrho(p,q)$ exhibits less involved behavior and we do not have to appeal to the heavy machinery used in \cite{JLR}. Consequently,  we obtain the sharp estimates much easily.

Before  stating our  result, we need to introduce some notations. Let $\mathfrak A,$ $ \mathfrak B$,  $\mathfrak C$, $\mathfrak  D$,  $\mathfrak  F\in [1/2,1]\times[0,1/2]$ be the points  defined   by
\begin{align*}
    &\mathfrak A=\left( \dfrac{2d+3}{2(2d+1)}, \  \dfrac12\right), 
    \  \  \mathfrak B=\left(\dfrac{(2d)^2+8d-1}{4d(2d+1)},\dfrac{2d-1}{4d}\right),
   \ \ \mathfrak C=\left(1,\frac{2d-1}{4d}\right),
   \end{align*}
   \begin{align*}   
   &\mathfrak D=\left(\frac{d+1}{2d},\frac{1}{2}\right),\  \ 
\quad \mathfrak F=\left(\frac{(2d)^2+4d-4}{4d(2d-1)},\frac{d-1}{2d-1}\right).
\end{align*}
For a point $(x,y)\in  [1/2,1]\times[0,1/2]$, set $(x,y)'=(1-y, 1-x)$ and, similarly, 
for a set $S\subset [1/2,1]\times[0,1/2]$  we set $S'=\{(x,y)':(x,y)\in  S\}$. Then, we define the set $\mathcal R_1$, $\mathcal R_2$, and  $\mathcal R_3\subset [1/2,1]\times[0,1/2]$ as follows.

\begin{defn} 
Let $\mathcal R_1$ denote the closed pentagon with vertices $(\frac12, \frac12),  \mathfrak A,  \mathfrak B,  \mathfrak B', \mathfrak A'$ from which  
two points $\mathfrak B$ and $\mathfrak B'$ are removed.  Let  $\mathcal R_2$ be the closed trapezoid with vertices $\mathfrak A , (1,\frac12), $ $\mathfrak C,$ $ \mathfrak B$ from which the closed line segment  $[\mathfrak B, \mathfrak C]$ is removed, and
$\mathcal R_3$ denote  the closed pentagon with vertices  $\mathfrak B, \mathfrak C, (1,0), \mathfrak C',$ and $\mathfrak B'$ from which the closed line segments $[\mathfrak B, \mathfrak C]$ and 
$[\mathfrak B', \mathfrak C']$ are removed.
 (See Figure \ref{fig1}).
\end{defn}

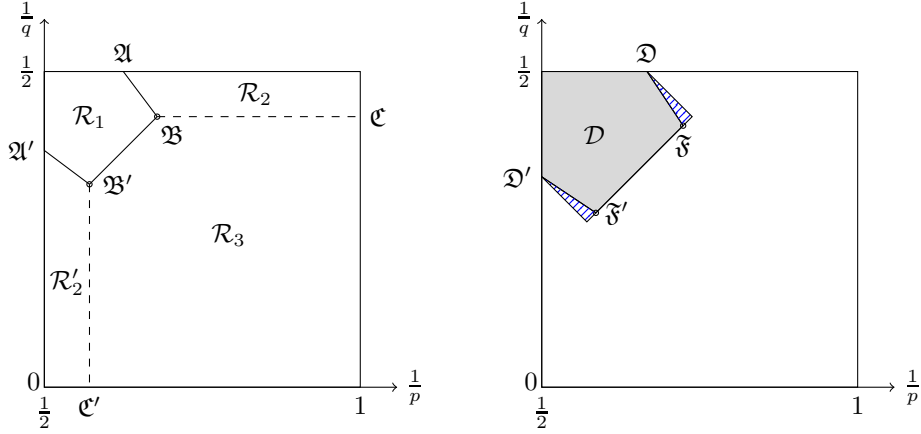
\begin{figure}[t]
\centering
\begin{tikzpicture}[scale=0.7]
\draw[<->] (0,7) node[left]{$\frac1q$}--(0,0)--(6.7,0)node[right]{$\frac1p$};
\draw (0,0) rectangle (6,6);
\draw (0,6)--(0,4.5)--(6/7,27/7) --(15/7,36/7)--(1.5,6);
\draw[dashed] (6/7,27/7)--(6/7,0); \draw[dashed] (15/7,36/7)--(6,36/7); 
\node[left] at (0,4.5) {$\mathfrak A'$}; \node[above] at (1.5,6) {$ \mathfrak A$}; 
\node[below] at (6/7,0) {$\mathfrak C'$}; \node[right] at (6,36/7) {$\mathfrak C$}; 
\node[] at (6/7,36/7) {$\mathcal R_1$};
\node[] at (3/7,14/7) {$\mathcal R_2'$}; \node[] at (28/7,39/7) {$\mathcal R_2$};
\node[above] at (3.5,2.5) {$\mathcal R_3$};
\draw (6/7,27/7) circle [radius=0.05]; \node[right] at (6.5/7,27/7) {$\mathfrak B'$};
\draw (15/7,36/7) circle [radius=0.05]; \node[below] at (17/7,36/7) {$\mathfrak B$}; 
\node[left] at (0.1,0.1) {$0$};\node[below] at (0,0) {$\frac12$};
\node[left] at (0,6) {$\frac12$}; 
\node[below] at (6,0) {$1$};
\end{tikzpicture}
\quad \quad 
\begin{tikzpicture}[scale=0.7]
\draw[<->] (0,7) node[left]{$\frac1q$}--(0,0)--(6.7,0)node[right]{$\frac1p$};
\draw[fill=gray! 30]  (0,6)--(0,4)--(36/35,116/35)--(94/35,174/35)--(2,6);
\draw (0,0) rectangle (6,6);
\node[] at (1,4.8) {$\mathcal D$};
\draw (0,4)--(36/35,116/35)--(94/35,174/35)--(2,6); 
\draw (36/35,116/35) circle [radius=0.05]; \node[right] at (36/35,116/35) {$\mathfrak F'$}; 
\draw (94/35,174/35) circle [radius=0.05]; \node[below] at(94/35,174/35)  {$\mathfrak F$}; 
\node[left] at (0,4) {$\mathfrak D'$}; \node[above] at (2,6) {$ \mathfrak D$}; 
\draw[pattern=north east lines, pattern color=blue] (0,4) -- (6/7,22/7)--(36/35,116/35);
\draw[pattern=north east lines, pattern color=blue] (2,6) -- (20/7,36/7)--(94/35,174/35);
\node[left] at (0.1,0.1) {$0$};\node[below] at (0,0) {$\frac12$};
\node[left] at (0,6) {$\frac12$}; 
\node[below] at (6,0) {$1$}; 
\end{tikzpicture}
\caption{The points $\mathfrak A,$ $\mathfrak B $, $\mathfrak C$, $\mathfrak D$, $\mathfrak F$, and the regions $\mathcal R_1$, $\mathcal R_2,$ $\mathcal R_3$, and $\mathcal D$. }
\label{fig1}
\end{figure}

For $(p,q)\in [1,2]\times [2,\infty]$,  we define the exponent  $\varrho(p,q)$ by setting 
\Be
\label{def-beta} 
\varrho(p,q)=
         \begin{cases} 
                 \  -\frac12(\frac1p-\frac1q),    & \ \big(\frac1p,\frac1q\big)\in\mathcal R_1,  
                                              \\  
                 \ d\big(\frac1p+\frac1q\big) -\frac{2d+1}{2}, & \ \big(\frac1p,\frac1q\big)\in\mathcal R_2, 
                    \\
               \     \frac{2d-1}{2}- d\big(\frac1p+\frac1q\big),& \ \big(\frac1p,\frac1q\big)\in\mathcal R_2',
                   \\
               \    d(\frac1p-\frac1q) -1,           & \ \big(\frac1p,\frac1q\big)\in \ol{\mathcal R}_3 .
\end{cases} 
\Ee

We are now  ready to state our main result.

\begin{thm}\label{main}  Let $d\ge 1$ and $1\le p\le 2\le q\le \infty$. We have the estimate
 \begin{align}
\label{est-main}    
\| \pj\|_{p\to q}\lesssim    {\mu}^{\varrho(p,q)}
\end{align}
if and only if $(1/p,1/q)\not\in [\mathfrak B,\mathfrak  C]\cup[\mathfrak B', \mathfrak  C']$. The bounds are  sharp in that  the exponents $\varrho(p,q)$ cannot be improved. Additionally,  we have
\begin{enumerate} 
[leftmargin=0.8cm, labelsep=0.3 cm, topsep=-6pt]
\item[$(i)\,$]
 If $(1/p,1/q)\in (\mathfrak B, \mathfrak C]$, we have 
$  \| \pj\|_{L^p\to L^{q,\infty}}\lesssim {\mu}^{\varrho(p,q)} $.
\item [$\,\,(ii)$]   If $(1/p,1/q)=\mathfrak B$ or $\mathfrak B'$, we  have  $
    \| \pj\|_{ L^{p,1}\to L^{q,\infty}} \lesssim   {\mu}^{\varrho(p,q)}$.  
    \end{enumerate}
\end{thm}
Here $\|\pj\|_{L^{p,r}\to L^{q,s}}$ means the operator norm of $\pj$ from the Lorentz space $L^{p,r}$ to $L^{q,s}.$

The twisted Laplacian is related to the Heisenberg sub-Laplacian \cite{KR,Th93,Th06}. The reduced Heisenberg group $h_d$ is   the set $\R^d\times \R^d\times \T$ with group law
\[ (x,y,e^{it}) (x',y,'e^{it'})=(x+x',y+y', e^{i(t+t'+\frac12(x'\cdot y -x\cdot y'))}),\] 
and the sub-Laplacian $\mathcal L$ on $h_d$ is defined by
\[ \mathcal L=-\sum_{j=1}^d \Big(\big(\frac\partial{\partial x_j}-\frac{1}2y_j \frac{\partial}{\partial t} \,\big)^2 +\big(\frac\partial{\partial{y_j}}+\frac{1}2x_j\dfrac{\partial}{\partial t} \,\big)^2\Big).\]
 The estimate \eqref{est-main} can be used to generate the spectral projection estimates for  the differential operators acting on a special class of functions on $h_d$.  Especially, 
on the class of functions of the form $g(x,y,t)=e^{imt} f(x,y)$, $m\in\Z$, we have $\mathcal L (e^{imt}f) =e^{imt} L_m f$, where 
\[ L_m =  - \sum_{j=1}^d \Big(\big(\frac\partial{\partial x_j}-\frac m2iy_j\big)^2 +\big(\frac\partial{\partial{y_j}}+\frac m2 ix_j\big)^2\Big).\]
By scaling, it is easy to see that, for each nonzero $m\in\Z$, the numbers $(2k+d)|m|$, $k\in \N_0$, are eigenvalues of $L_m$ with  the corresponding eigenfunctions 
\[ 
\Phi_{\alpha,\beta}^m(x,y)=|m|^{d/2} \Phi_{\alpha,\beta}(|m|^{1/2}x,{\sgn(m) } |m|^{1/2}y),
\]
which form an orthonormal basis of $L^2(\R^{2d}).$ So, the pairs $(|m|(2k+d), m)$, $m\in\Z\setminus\{0\}, k\in\N_0$, give the discrete joint spectrum of $\mathcal L$ and $-i\partial_t$. Let $\mathcal P_{m,k}$ be the projection onto the joint eigenspace corresponding to the eigenvalue $(|m|(2k+d), m)$ (see \cite{Th93, KR} for further details). Then the spectral projection estimate \eqref{est-main}  yields
\[ \| \mathcal P_{m,k} u\|_{L^q(h_d)}\lesssim  (2k+d)^{\varrho(p,q)}|m|^{d(\frac1p-\frac1q)}\|u\|_{L^p(h_d)}.\]

We now consider  the estimate for the resolvent of $L$ which takes the form
\begin{equation}\label{resol_z}
\|(L-z)^{-1}\|_{p\to q}\le C_z,\quad z\in \C\setminus (2\N_0+d), 
\end{equation}
where $(L-z)^{-1}$ is defined by
\begin{equation}\label{resol_de}
 (L-z)^{-1}f=\sum_{\mu} ({\mu}-z)^{-1} \pj f.
 \end{equation}
 Estimates for  resolvents  have  a wide range of applications. 
 In particular, uniform resolvent estimates for partial differential operators which holds with $C_z$ independent of the spectral parameter have been studied related to Carleman estimate and strong/weak unique continuation properties (for example, see \cite{KRS87, ev01, KT09, JKL16, JLR} 
 and references therein). 
For the closely related Hermite operator $H$, it was shown in \cite{ev01, KT09,JLR} that 
\begin{equation}\label{hermite_r}
 \|(H-z)^{-1}\|_{p\to q}\le C
 \end{equation}
under the  spectral gap condition 
\begin{equation}\label{gap}
\dist(z, 2\N_0+d)\ge c
\end{equation}
for some $c>0.$ Escauriaza and Vega proved  \eqref{hermite_r} for $\frac{2d}{d+2}\le p\le 2\le q\le \frac{2d}{d-2}$ and showed  strong unique continuation property for the differential inequality $|\partial_t u +\Delta u|\le |Vu|$ with $V\in L^\infty_t L^{d/2}_x$. Also, in \cite{JLR}, the authors extend the range of $p,q$ for \eqref{hermite_r} to an interval on $1/p-1/q=2/d$ and obtained the strong unique continuation property with $V\in L^\infty_tL^{d/2,\infty}_x.$ 

As an application of the spectral projection estimate we obtain the following uniform resolvent estimate for the twisted Laplacian.

\begin{thm}\label{resol} 
Let $d\ge 2$. Suppose that  $(1/p,1/q)$ is in the closed pentagon $\mathcal D$ with vertices $(1/2,1/2), \mathfrak D, \mathfrak F, \mathfrak F',\mathfrak D'$ from which $\mathfrak F$ and $\mathfrak F'$ are removed (the gray region in Figure \ref{fig1}). Then there is a constant $C>0$ such that  
\Be
\label{l-resolvent}
\|(L-z)^{-1}\|_{p\to q}\le C  
\Ee
 provided that  $ z\in\C$ satisfies  \eqref{gap} for some $c>0$. Furthermore, if $\pq=\mathfrak F$ or $ \mathfrak F'$, we have   the restricted weak type estimate  $\|(L-z)^{-1} f\|_{q,\infty}\le C \|f\|_{p,1} $ provided that \eqref{gap} holds for some $c>0$.
\end{thm}

It is natural to expect that,  as an application of the uniform resolvent estimates, one may be able to show strong unique continuation property for the heat equation associated with $L$ as  in the previous works but we do not intend to pursue the matter here. 
When $p=q'$, the estimate \eqref{l-resolvent} was previously  obtained by  Cuenin \cite[Proposition 2.2]{Cu17} to show clustering estimate for eigenvalues of the twisted Laplacian with $L^p$ potentials. In fact, he obtained the resolvent estimate \eqref{resol_z} with bound $ (1+|\re z|)^{\varrho(q',q)}(1+ \delta(z)^{-1})$ where $\delta(z):=\dist (z, 2\N_0+d)$.


As in the case of Hermite resolvent estimate, the gap condition \eqref{gap} is necessary for the uniform estimate \eqref{l-resolvent} because the twisted Laplacian has discrete eigenvalues. Indeed, $\| (L-z)^{-1}\|_{p\to q} =|{\mu} -z|^{-1} \|f\|_q/\|f\|_p$  if $f$ is in the eigenspace corresponding to ${\mu}\in 2\N_0+d.$ Thus, the operator norm  goes to infinity as $z$ towards ${\mu}.$ 

By making use of   Theorem \ref{fractional}, one can easily show \eqref{l-resolvent} holds only for $1/p-1/q\le  1/d$. Thus,  {$1/p-1/q=1/d$} is the critical case for the estimate \eqref{l-resolvent}, and  it is  more difficult  to obtain the uniform estimate \eqref{l-resolvent} for $p,q$ satisfying $1/p-1/q=1/d$. As for the critical case, we establish \eqref{l-resolvent}  for  $(1/p,1/q)\in (\mathfrak F, \mathfrak F')$ in Theorem \ref{resol}. However, we could not obtain fully expected result. More precisely, from \eqref{resol_de} and the estimate for the  fractional twisted Laplacian operator (Theorem \ref{fractional}), one may expect that the uniform resolvent estimate \eqref{l-resolvent} holds for any $p,q$ for which the uniform spectral projection estimate $\|\pj\|_{p\to q}\lesssim 1$ holds. But there is a gap between the ranges of $p,q$ for which \eqref{est-main} holds with $\varrho(p,q)\le 0$ and \eqref{l-resolvent} holds (see the slashed region in Figure \ref{fig1}).

The rest of the paper is organized as follows. In Section 2 we provide a representation formula for $\pj$  which will be useful to show the sharp $L^p$--$L^q$ estimate for $\pj$. We separately prove the sufficiency and the necessary parts of  Theorem \ref{main} in Section 3 and Section 4. We provide  the proof of the uniform resolvent estimate for $L$  in Section 5.

\section{Preliminaries}

\subsection{Representation formula  for $\pj$} 
The Schr\"odinger propagator $e^{itL}$  associated with $L$ can be expressed  by using the spectral decomposition of $L$, that is to say, 
\[
\label{schrodinger} e^{itL} f= \sum_{\mu}   e^{it \mu} \pj f, \quad t\in \mathbb R.
\]
So, we clearly have 
\Be
\label{iso}  
\|e^{itL} f\|_2=\|f\|_2, \quad t\in\mathbb R.
\Ee
Since  the  eigenvalues of $L$ are in $2\mathbb N_0+d$, the difference of two eigenvalues $\mu, \mu'$ of  $L$ is in  $2\Z$, i.e., ${\mu} -{\mu}'\in 2\Z$. As in the case of the
Hermite spectral projection, $\pj$ is also written as follows: 
\Be 
\label{represent0} 
\pj f(z)=\frac1{\pi}\int_{-\pi/2}^{\pi/2} e^{it\mu} e^{-it L} f(z) dt,  \quad f\in \mathcal S(\R^{2d}).
\Ee
Set $z=x+iy$ and $z'=x'+iy'\in \C^d\cong \R^{2d}.$  The same idea of exploiting the specific form of the eigenvalues was already used in \cite{JLR}. We note that  the Schr\"odinger propagator $e^{-itL}$ also has the following kernel representation: 
\Be \label{dispersive} e^{-itL}f(z)=C_d    (\sin t)^{-d}\int e^{i ( \frac{|z-z'|^2}{4}\cot t +\frac12 \im z\cdot \overline{z}' )} f(z')\,dz', \Ee
where $C_d$ is a constant depending only on $d$. This can be easily deduced from the corresponding  kernel formula for the heat operator $e^{-tL}$ by replacing  $t$  with $it$ (see \cite[p.37]{Th93}).  Since $\sum_{\mu} \pj f$ converges absolutely and uniformly for $f\in \mathcal S(\R^{2d})$ (see \eqref{basic}),   we now get
\begin{align}
 \label{represent} 
 \pj f(z)= C_d  \int_{\C^{d}}\int_{-\pi/2}^{\pi/2} (\sin t)^{-d}e^{i (t {\mu} + \frac{|z-z'|^2}{4}\cot t +\frac12 \im z\cdot \overline{z}' )} f(z') dt\,dz'
 \end{align}
for any $f\in \mathcal S(\R^{2d})$.

Since  the kernel of $e^{it({\mu}-L)} $ has a singularity at $t=0$, we need to decompose it away from the singularity.  For any function $\eta\in C^\infty(\R)$, let us define $\pj[\eta]$ by 
\[\pj [\eta]f:=\frac1\pi\int_{-\pi/2}^{\pi/2} \eta(t) e^{it\mu} e^{-it L}f dt. \]
Let $\psi\in C^\infty_c([\frac14,1])$ be a smooth function such that  $\sum_{j\in \Z} \psi(2^jt)=1$ for all $t>0$. We choose $\psi^0$ so that  
\begin{equation} 
\label{partitionof1}\psi^0(t) + \sum_{j=3}^\infty \Big(\psi(2^jt)+\psi(-2^jt)\Big) =1 \end{equation}
for any $t\in[-\frac\pi2,\frac\pi2]$. 
Clearly $\psi^0$ is smooth  on $(-\frac\pi 2,\frac\pi2)$ and continuous and symmetric on $[-\frac\pi2,\frac\pi2]$ and we can  extend $\psi^0$ periodically to the whole real line with period $\pi$, which is also smooth. Using the partition of unity, we decompose the projection operator  as follows: 
\begin{equation}
\label{primary-decomp}
\pj  = \pj[\psi^0] + \sum_{j=3}^\infty  \Big(\pj[\psi^+_j ]+\pj[\psi^{-}_j ]\Big) 
\end{equation}
for any Schwartz function $f\in \mathcal S(\R^{2d})$ where 
$\psi^\pm_j(t)=\psi(\pm2^jt)$. 
We make further decomposition of $\pj[\psi^0]$ by breaking $\psi^0$ away from 
$\pi/2$ where the second derivative of the phase function vanishes. 
For the purpose  we denote $\wpsi=\psi(|\cdot|)$ 
and  define periodic functions $ \varphi^0$,  $\varphi_k$ of period $\pi$ by setting
\begin{equation}\label{cutoff}
\begin{aligned}
\quad \varphi_k(t)&=\psi^0(t) \wpsi(2^k(t-\pi/2)),  \\
 \varphi^0(t)&=\psi^0(t)\Big(1-\sum_{k\ge 5} \wpsi(2^k(t-\pi/2))\Big)
 \end{aligned}\end{equation}
 for $t\in (0, \pi)$. 
Hence, we get 
\begin{equation}\label{decom1}
\pj =\sum_{\pm}\sum_{j\ge 3}\pj[\psi^\pm_j] + \sum_{k\ge5}\pj[\varphi_k] + \pj[\varphi^0] 
.\end{equation}
 Since the eigenvalues of $L$ are in $2\N_0+d$, as before it is clear from spectral decomposition  that 
\begin{equation}\label{decom} e^{it(L-{\mu})} =e^{i(\pi+t)(L-{\mu})} .  \end{equation} 
 Thus, by periodicity it follows that 
 $\pj[\eta] f=\frac1{\pi} 
\int_{0}^{\pi}\eta(t) e^{it(L-{\mu})} f dt$
for any $\pi$ periodic $\eta$.  In particular, we have 
\Be
\label{proj-varphi}
\pj[\varphi_k]=\frac1{\pi} 
\int_{0}^{\pi}\varphi_k(t) e^{it(L-{\mu})} f dt, \quad  {k= 5, 6, \dots}.
\Ee

\subsection{Estimate for oscillatory integral}
Let the phase function $\phi$ be defined by
$$\phi(t):=\phi(z,z',t): = t +\frac{ |z-z'|^2}4 \cot t  +\frac12\im (z\cdot \overline z'),\quad z,z'\in\C^d.$$
We also define the oscillatory integrals $\CI_j$, $\CJ_k$ and $\CJ^0$ by 
\begin{align*}
\CI_j({\mu}) &:= \CI_j(z,z',{\mu}):=\int \eta(2^j t) e^{i{\mu} \phi(z,z',t)} dt, \\
\CJ_{k}({\mu}) &:= \CJ_{k}(z,z',{\mu}):=\int \psi^0( t)\eta(2^k(t-\pi/2)) e^{i{\mu} \phi(z,z',t)} dt,\\
\CJ^{0}({\mu}) &:= \CJ^{0}(z,z',{\mu}):=\int_0^\pi \varphi^0( t) e^{i{\mu} \phi(z,z',t)} dt,\end{align*}
for $j, k\in \Z$, ${\mu}\in\R$, and  $z,z'\in \C^d$ where $\eta$ is a  function supported in $[-1,-1/4]\cup [1/4,1]$.  
In what follows we show the estimates for ${\CI_j(\mu)}$, $\CJ_{k}({\mu})$, and $\CJ^{0}({\mu})$ which are crucial  for obtaining the sharp estimates for $\pj$.

\begin{lemma}\label{oscill}
 Let $d\ge1$, $j,k\ge 1$. Let $\eta$ be a function supported in $[-1,-1/4]\cup [1/4,1]$ and $|\frac{d^l}{dt^l} \eta(t)|\lesssim 1$, $l=0,1$. Then we have
\begin{align}
 |\CI_j(z,z',{\mu})|&\le C{\mu}^{-1/2}2^{-j/2},\label{Ij}\\
 |\CJ_{k}(z,z',{\mu})|&\le C{\mu}^{-1/2}2^{k/2},\label{Jk}\\
 |\CJ^0(z,z',{\mu})|&\le C{\mu}^{-1/2},\label{J0}
 \end{align}
with $C$ independent of $z,z'\in \C^d$, $j,k$, and ${\mu}>1.$
\end{lemma}

\begin{proof}  
To show {\eqref{Ij}--\eqref{J0}}, we make use of the well-known van der Corput's lemma (see, for example \cite[p.334]{Stein}).
We consider the time derivative of the phase function $\phi$ of the integrals {$\CI_j$, $\CJ_k$ and $\CJ^0$}. A simple computation shows 
\begin{equation}\label{phi'}
 \phi'(t) = \frac{4\sin^2 t-|z-z'|^2}{4\sin^2t}.\end{equation}
We first show \eqref{Ij} for $j\ge 1.$ If $|z-z'|\ge2$, there is no critical point of $\phi$ on $\supp \eta(2^j \cdot)$ because $\eta$ is supported in $[-1,-1/4]\cup [1/4,1]$. So,  it is easy to see that
\[ |\phi'(t)|\gtrsim 2^{2j} |(2\sin t-|z-z'|)(2\sin t+|z-z'|)|\gtrsim 2^{2j}\]
on the support of $\eta(2^j\cdot)$. Thus, applying  van der Corput's lemma yields 
\[ |\CI_j({\mu})|\lesssim \min\{ ({\mu} 2^{2j})^{-1},2^{-j}\}\lesssim {\mu}^{-1/2}2^{-3j/2}.\]
So, we may assume $|z-z'|<2$. If $|z-z'|>2^{3-j}$ or $|z-z'|<2^{-4-j}$,  by \eqref{phi'} 
we have $|\phi'(t)|\gtrsim 2^{2j}\max\{2^{-2j}, |z-z'|^2\}\gtrsim 1$. Hence,  by the van der Corput lemma we have $ |\CI_j({\mu})|\lesssim \min\{{\mu}^{-1}, 2^{-j}\}\lesssim {\mu}^{-1/2}2^{-j/2}.$ To complete the proof of \eqref{Ij}  we only need to consider the case 
\[ |z-z'|\sim 2^{-j}.\] 
Let us note that 
\Be
\label{second-d}
\phi''(t)=2\cos t|z-z'|^2 (\sin t)^{-3}.
\Ee
and $|\phi''|\gtrsim 2^j$
on the support of $\eta(2^j \cdot )$.  Applying van der Corput's lemma again, we have \eqref{Ij}. 
This completes the proof of \eqref{Ij}.

We next show the estimate \eqref{Jk}  for $k\ge1$. If $|z-z'|\le 1/2$, we have $|\phi'(t)|\gtrsim 1$ on the support of $\psi^0$ because  $2|\sin t|\ge 1$. Thus  $|\CJ_{k}({\mu})|\lesssim {\mu}^{-1}\le {\mu}^{-1/2}2^{k/2}.$ 
We may now assume $|z-z'|>1/2$. Using \eqref{second-d}, we have  
\[ |\phi''(t)|\gtrsim |\cos t|=|\cos t-\cos( \pi/2)|\gtrsim 2^{-k} \]
on $\supp \psi^0(\cdot) \eta(2^k(\cdot - \pi/2))$.
Thus, van der Corput's lemma gives the desired result \eqref{Jk}.

Finally,  noting  that $\dist(\supp \varphi^0, \{0,\pi/2,\pi\})\ge c$ for some $c>0$ because of \eqref{cutoff}, we see  
{that  $|\phi'(t)|\gtrsim 1$ if $|z-z'|\le 1/2$  and $|\phi''(t)|\gtrsim 1$}  {if $|z-z'|\ge1/2$ on the support of $\varphi^0$}. Hence, the estimate  \eqref{J0} follows from  the van der Corput lemma.
\end{proof}

We frequently  make use of the following summation trick  to handle  the endpoint cases \cite{bak, Car99}.

\begin{lemma}\label{s-trick} \cite[Lemma 2.4]{JLR}   Let $1\le p_l,  q_l\le \infty$ and $\epsilon_l>0$ for $l=0,1$,  and  set $\theta=\frac {\epsilon_0}{\epsilon_0+\epsilon_1}$, 
$\frac1p_\ast=\frac \theta{p_1}+ \frac {1-\theta}{p_0},$ and $\frac1q_\ast=\frac \theta{q_1}+\frac {1-\theta}{q_0}. $  Suppose that  $T_j$, $j\in \mathbb Z $  are sublinear operators defined from $L^{p_l}\to  L^{q_l}$ with 
\[ \|T_j\|_{p_l\to q_l}\le B_l 2^{ j (-1)^l \epsilon_l }, \quad l=0,1.\]   
Then we have the following.
\vspace{-6pt}
\begin{enumerate} 
[leftmargin=.65cm, labelsep=0.15 cm, topsep=0pt]
\item[$(a)$] If $p_0=p_1=p$ and $ q_0\neq q_1$, then   $\|\sum_j T_j f  \|_{L^{q_\ast,\infty}}\lesssim B_0^{1-\theta}B_1^\theta \|f\|_{p}$. 
\item[$(b)$] If $q_0=q_1=q$ and $ p_0\neq p_1$, then $\|\sum_j T_j f  \|_{L^{q}}\lesssim B_0^{1-\theta}B_1^\theta \|f\|_{p_\ast,1}$.
\item[$(c)$] If $p_0\neq p_1$ and $ q_0\neq q_1$, then  $\|\sum_j T_j f  \|_{L^{q_\ast,\infty}}\lesssim B_0^{1-\theta}B_1^\theta \|f\|_{p_\ast, 1}$.
\end{enumerate}
\end{lemma} 

We close this section with some sharp $L^1$--$L^2$, $L^1$--$L^\infty$ estimates for $\pj[\eta_k]$, which are useful in  showing the weak type estimate for $\pj$ at $(1/p,1/q)\in (\mathfrak B,\mathfrak C].$ 
\begin{lemma}\label{local} Let $d\ge 1$, ${\mu}\in 2\N_0+d$, and $k\in \N_0$. 
Suppose  $\eta_k \in C^\infty_c([-\pi/2,\pi/2])$  
such that  $\supp \eta_k$ is contained in an interval of length $\sim 2^{-k}$ and satisfies $|\frac{d^l}{dt^l}\eta_k(t)|\le C2^{kl}$ for all $l\in \N_0$. If $2^{k}\lesssim {\mu}$, then  we have
\begin{align}
\|\pj [\eta_k]\|_{1\to 2}&\lesssim 2^{-k/2}{\mu}^{\frac{d-1}2},\label{12}\\
\|\pj [\eta_k]\|_{1\to \infty}&\lesssim {\mu}^{d-1} .\label{1infty}
\end{align}
\end{lemma}
\begin{proof} We prove \eqref{12} and \eqref{1infty} by combining with  the known $L^1$--$L^2$ estimate for $\pj$ (\eqref{known}) and  the spectral decomposition.  Note that 
\begin{equation}\label{spectral}
 \pj[\eta_k] f =\sum_{{\mu}'} \frac1{\pi}\int_{-\pi/2}^{\pi/2} \eta_k(t) e^{it({\mu}-{\mu}')}P_{{\mu}'} fdt=
 \frac1{\pi}\sum_{{\mu}'} \widehat \eta_k({\mu}'-{\mu})P_{{\mu}'}f.\end{equation}
 By orthogonality and the estimate 
 $\|\pj\|_{1\to 2}\lesssim \mu^{\frac{d-1}2}$ (\eqref{known}) we see that 
 \begin{align*}
\|\pj f\|_2^2\lesssim  \sum_{{\mu}'} |\widehat \eta_k({\mu}'-{\mu})|^2 \|P_{{\mu}'} f\|^2_2&\le C\sum_{{\mu}'}|\widehat \eta_k({\mu}'-{\mu})|^2({\mu}')^{d-1} \|f\|^2_1.
\end{align*}
Since $|\widehat \eta_k(t)|\le C_N 2^{-k}(1+2^{-k}|t|)^{-N}$ for any $N$ with $C_N$ independent of $k$ and since 
 $2^{k}\lesssim {\mu}$, 
 we have
\begin{align*}
\|\pj f\|_2^2
&\lesssim \sum_{{\mu}'}2^{-2k}(1+2^{-k}|{\mu}'-{\mu}|)^{-N}({\mu}')^{d-1}\|f\|^2_1
\lesssim 2^{-k}{\mu}^{d-1}\|f\|_1^2,
\end{align*}
which yields \eqref{12}.    The estimate \eqref{1infty} can be shown  in the same manner using \eqref{spectral} since  we have $\|P_{{\mu}}f\|_\infty\lesssim {\mu}^{d-1}\|f\|_1$  by \eqref{known} and duality.  We omit the detail. 
\end{proof}

\section{Proof of Theorem \ref{main}: Sufficiency part}
In this section, we show \eqref{est-main} for $p,q$ satisfying $1\le p\le 2\le q\le \infty$ and   $(1/p,1/q)\not\in [\mathfrak B, \mathfrak C]\cup[\mathfrak B', \mathfrak C']$ and  obtain the weak/restricted-weak type estimates for $\pj$ for  $(1/p,1/q)\in [\mathfrak B, \mathfrak C]\cup[\mathfrak B', \mathfrak C']$. Our argument here is similar with  the one  used in the proof of the local estimate for the Hermite spectral projection  (Theorem 1.5 of \cite{JLR}).

From the known $L^2$--$L^q$ bound for $\pj$(\eqref{known}) and duality, we already have \eqref{est-main} when $p=2$, $q=2$, or $p=q'.$ Thus, by duality and interpolation, it suffices to show \eqref{est-main} for $(1/p,1/q)\in\mathcal R_1$, the  weak type estimate $\|\pj \|_{L^{p}\to L^{q,\infty}}\lesssim {\mu}^{\varrho (p,q)}$ for $(1/p,1/q)\in(\mathfrak B,\mathfrak C]$ (the assertion $(i)$), and the restricted-weak type estimate $\|\pj \|_{L^{p,1}\to L^{q,\infty}}\lesssim {\mu}^{\varrho(p,q)}$ at $(1/p,1/q)= \mathfrak B$ (the assertion $(ii)$).

\subsection*{Strong type  estimate for $\pj$ when $\ppq\in \mathcal R_1$}
We first prove \eqref{est-main} for $(\frac1p,\frac1q)\in \mathcal R_1$. In view of \eqref{decom1} and Lemma \ref{s-trick}, it is enough to show that for $j\ge3$ and $k\ge5$ 
\begin{align}
\|\pj[\psi_j^\pm]\|_{p\to q} &\lesssim {\mu}^{-\frac12(\frac1p-\frac1q)} 2^{j(-1+\frac{2d+1}2(\frac1p-\frac1q))},\label{TT*}\\
\|\pj [\varphi_k]\|_{p\to q}&\lesssim {\mu}^{-\frac12(\frac1p-\frac1q)} 2^{k(-1+\frac{2d+1}2(\frac1p-\frac1q))},\label{TT*k}\\
\|\pj[\varphi^0]\|_{p\to q} &\lesssim {\mu}^{-\frac12(\frac1p-\frac1q)},\label{TT*0}
\end{align} 
whenever $\ppq$ is in the closed quadrangle $\mathcal Q(d)$ with vertices $(\frac12,\frac12)$, $\mathfrak A$, $(1,0)$, and $\mathfrak A'$.
Indeed, by \eqref{decom1}, \eqref{TT*}, {\eqref{TT*k}} and the triangle inequality, we obtain
\begin{align*}
 \|\pj\|_{p\to q}&\lesssim  \sum_\pm \sum_{j\ge 3}\|\pj[\psi_j^\pm]\|_{p\to q}+\sum_{k\ge5}\|\pj [\varphi_k]\|_{p\to q} +\|\pj[\varphi^0]\|_{p\to q}\lesssim {\mu}^{-\frac12(\frac1p-\frac1q)}\end{align*}
if $\ppq\in \mathcal Q(d)$ satisfying $\frac1p-\frac1q <\frac2{2d+1}$. For $ \ppq=\mathfrak B$ or $\mathfrak B' \in \mathcal Q(d)$, which satisfies $\frac1p-\frac1q =\frac2{2d+1}$,   $(c)$ in  Lemma \ref{s-trick}  implies
\[ \|\pj\|_{L^{p,1}\to L^{q,\infty}}\lesssim {\mu}^{-\frac12(\frac1p-\frac1q)}.\]
Moreover, this shows the assertion $(ii)$ in Theorem \ref{main}.
By  real interpolation between the restricted-weak type $(p,q)$ estimates with $\ppq=\mathfrak B$ and $\mathfrak B'$, we get \eqref{est-main} for $\ppq \in(\mathfrak B,\mathfrak B')$ and, hence,  for all $\ppq\in \mathcal R_1$.  

As to be seen later,  better bounds  are possible for $\pj[\varphi_k]$ and $\pj[\varphi^0]$, but   \eqref{TT*k} and \eqref{TT*0} are sufficient for our purpose.

We now show \eqref{TT*}--\eqref{TT*0} for $\ppq\in \mathcal Q(d).$ 
Thanks to \eqref{schrodinger}  we clearly have the isometry $\|e^{-it L}f\|_2=\|f\|_2$.   
It is clear that  the modified $TT^*$-argument  in \cite[Lemma 2.3]{JLR} works without modification. 
 From \eqref{Ij}, we have  $\|\pj [\eta_j]\|_{1\to \infty}\lesssim {\mu}^{-1/2}2^{j(d-1/2)}$ whenever $\eta_j$ is a smooth function supported in $[-2^{-j},-2^{-j-2}]\cup[2^{-j-2},2^{-j}]$ and satisfies $|\frac{d^l}{d t^l} \eta_j(t)|\le C2^{jl}$ for $l=0,1,2$.   Thus,  \cite[Lemma 2.3]{JLR} gives the estimate \eqref{TT*}
for   $j\ge3$ and  $(\frac1p,\frac1q)\in\mathcal Q(d)$. 

 We next consider the estimate  \eqref{TT*k} for $\pj[\varphi_k]$. Here, the cutoff function $\varphi_k$ is supported near $\frac \pi2$. So,  Lemma 2.3 in \cite{JLR} does not apply directly, but a little modification of the argument gives  the desired result. Since $|\sin t|\gtrsim 1 $ on the support of $\varphi_k$, from \eqref{Jk}, we have $\|\pj [\varphi_k]\|_{1\to \infty} \lesssim {\mu}^{-1/2}2^{k/2}$  for any $k\ge5$. Taking interpolation with  $\|\pj [\varphi_k]\|_{2\to2}\lesssim 2^{-k}$ which follows from the isometry \eqref{iso} and Minkowski's inequality,  we get 
\[ \|P_{{\mu}}[\varphi_k]\|_{p\to p'}\lesssim {\mu}^{-\frac12(\frac1p-\frac1{p'})} 2^{k(-1+\frac32(\frac1p-\frac1{p'}))}\le {\mu}^{-\frac12(\frac1p-\frac1{p'})} 2^{k(-1+\frac{2d+1}2(\frac1p-\frac1{p'}))}\]
 for $1\le p\le 2$. 
Thus,  in order to show \eqref{TT*k}  for $\ppq\in \mathcal Q(d)$, by interpolation and duality it suffices to show \eqref{TT*k} with  $\ppq=(1/p_\circ,1/q_\circ):=\mathfrak A$, i.e.,  $q_\circ=2$ and $p_\circ=2(2d+1)/(2d+3)$. Equivalently, we will show that
\begin{equation}\label{tem1}
 \|\pj[\varphi_k]^*\pj[\varphi_k]\|_{p \to q}\lesssim  {\mu}^{-\frac12(\frac1{p}-\frac1{q})}2^{-k} \end{equation}
 with $(p,q)=(p_\circ, p_\circ')$. 
By a simple change of variables, we see that 
\[ \pj[\varphi_k]^*\pj[\varphi_k] f= \int \varphi_k(s)\int \varphi_k(t+s)e^{it({\mu}-L)} fdt\,ds.\]
Here, we note that  the support of $\varphi_k(\cdot+s)$ is contained in $[-2^{-k+1},2^{-k+1}]$ for any $s\in \supp \varphi_k$. Let us define $(\pj[\varphi_k]^*\pj[\varphi_k])_l $ for $l\in\Z$ by 
\[ (\pj[\varphi_k]^*\pj[\varphi_k])_l f:=\int \varphi_k(s)\int \varphi_k(t+s)\wpsi(2^l t)e^{it({\mu}-L)} fdtds  \]
and we may write
\[ \pj[\varphi_k]^*\pj[\varphi_k]  =\sum_{l\ge k-2} (\pj[\varphi_k]^*\pj[\varphi_k])_l  .\]
Note that $l\ge 3$ and  the estimate \eqref{TT*} is valid with $\psi_j^\pm$  replaced by a smooth function $\eta_j$ supported in $[-2^{-j},-2^{-j-2}]\cup[2^{-j-2},2^{-j}]$ and satisfies $|\frac{d^n}{d t^n} \eta_j(t)|\le C2^{jn}$ for $n=0,1,2$. Applying \eqref{TT*} to  $\pj[\varphi_k(\cdot+s)\wpsi(2^l\cdot)]$, 
we have 
\[ \| (\pj[\varphi_k]^*\pj[\varphi_k])_l\|_{p\to q} \lesssim {\mu}^{-\frac12(\frac1p-\frac1q)}2^{-k} 2^{l(-1+\frac{2d+1}2(\frac1p-\frac1q))}\]
for all $(1/p,1/q)\in\mathcal Q(d)$.
As before, Lemma \ref{s-trick} gives
\[\| \pj[\varphi_k]^*\pj[\varphi_k]  \|_{L^{p,1}\to L^{q,\infty}}\lesssim {\mu}^{-\frac12(\frac1p-\frac1q)}2^{-k} \]
for $(1/p,1/q)\in\mathcal Q(d)$ satisfying $1/p-1/q=2/(2d+1)$. Real interpolation yields \eqref{tem1} for $ (1/p,1/q)\in(\mathfrak B,\mathfrak B').$ 
Since $(1/p_\circ,1/p_\circ')$ is the intersection between the line segment $(\mathfrak B,\mathfrak B')$ and the line of duality, we get the desired estimate \eqref{tem1}  with $(p,q)=(p_\circ, p_\circ')$.

It remains to show the estimate for $\pj[\varphi^0]$. As before, $|\sin t|\gtrsim 1$ on the support of $\varphi^0$. So, from \eqref{J0} we have $\| \pj  [\varphi^0]\|_{1\to \infty} \lesssim \mu^{-1/2}.$ Interpolating this with the $L^2$ estimate derived from the $L^2$ isometry of $e^{itL}$, we obtain for $1\le p\le 2$
\begin{equation}\label{tempp'}
\| \pj [\varphi^0]\|_{p\to p'}\lesssim \mu^{-\frac12(\frac1p-\frac1{p'})}.
\end{equation}
Thus, in view of interpolation, it is enough to show \eqref{TT*0} for $(1/p,1/q)=\mathfrak A$. Equivalently, we will show 
\[
\| \pj [\varphi^0]^* \pj [\varphi^0]\|_{p\to q}\lesssim \mu^{-\frac12(\frac1{p}-\frac1{q})}
\]
for $(p,q)=(p_\circ,p_\circ')$ with $p_\circ=2(2d+1)/(2d+3).$ Actually, it follows from the known bounds \eqref{TT*}, \eqref{TT*k} and \eqref{tempp'}.
Indeed, from the periodicity of $\varphi^0$ and \eqref{decom} it is easy to see that
\begin{align*}
  \pj [\varphi^0]^* \pj [\varphi^0] f 
 &=\int_{-\frac \pi2}^{\frac \pi 2}\varphi^0(s)
 \int_{-\frac \pi2}^{\frac \pi 2}\varphi^0(t+s) e^{it(\mu-L)}f dt ds.
 \end{align*}
 So, using this and the partition of unity \eqref{partitionof1} we note that $\pj [\varphi^0]^* \pj [\varphi^0] $ equals
\begin{align*}
\int_{-\frac \pi2}^{\frac \pi 2}\varphi^0(s)
\Big\{ \sum_\pm\sum_{j\ge 3}  \pj [\varphi^0(\cdot +s)\psi_j^\pm] +\sum_{k\ge 5}  \pj [\varphi^0(\cdot +s)\varphi_k]+ \pj [\varphi^0(\cdot +s)\varphi^0]\Big\} ds.
\end{align*}
We have already had estimates for $ \pj [\varphi^0(\cdot +s)\psi_j^\pm] $ and $ \pj [\varphi^0(\cdot +s)\varphi_k]$ (\eqref{TT*},\eqref{TT*k}). Thus, Lemma \ref{s-trick} and the real interpolation imply 
\[ \| \sum_\pm\sum_{j\ge 3}  \pj [\varphi^0(\cdot +s)\psi_j^\pm] +\sum_{k\ge 5}  \pj [\varphi^0(\cdot +s)\varphi_k]\|_{p\to q}\lesssim \mu^{-\frac12(\frac1p-\frac1q)}\]
for $(1/p,1/q)\in(\mathfrak B, \mathfrak B')$. Since $(1/p_\circ, 1/p_\circ')$ is in $(\mathfrak B, \mathfrak B')$, this particularly yields
\[ \Big\| \int_{-\pi/2}^{\pi/2} \varphi^0(s)\Big(\sum_\pm\sum_{j\ge 3}  \pj [\varphi^0(\cdot +s)\psi_j^\pm] +\sum_{k\ge 5}  \pj [\varphi^0(\cdot +s)\varphi_k]\Big) ds\Big\|_{p_\circ\to p_\circ'}\lesssim \mu^{-\frac12(\frac1{p_\circ} -\frac1{p_{\circ}'})}.\] 
Moreover, from \eqref{tempp'}, we also have 
\[ \Big\| \int_{-\pi/2}^{\pi/2} \varphi^0(s)  \pj [\varphi^0(\cdot +s) \varphi^0]  ds\Big\|_{p_\circ\to p_\circ'}\lesssim \mu^{-\frac12(\frac1{p_\circ} -\frac1{p_{\circ}'})}.\] 
Combining these two estimates, we obtain the desired estimate.

\subsection*{Weak type  estimate for  $(1/p,1/q)\in (\mathfrak B,\mathfrak C]$}  
Recalling \eqref{decom1},  we first handle $\sum_{j\ge 3} \pj [\psi^\pm_j]$. To obtain the weak type $(p,q)$ estimate for $\pj$, we prove 
\begin{equation}\label{tem_goal} \| \sum_{j\ge 3} \pj [\psi^\pm_j] \|_{L^p\to L^{q,\infty}} \lesssim  {\mu}^{ d(\frac1p+\frac1q)-\frac{2d+1}2}\end{equation}
 for $(1/p,1/q)\in (\mathfrak B,\mathfrak C].$  
Since $  \sum_{2^{-j}\lesssim {\mu}^{-1}}\psi(2^jt) = \zeta({\mu} t) $ for some $\zeta\in C^\infty_c((0,\pi/2))$, 
we may write 
\begin{align*} 
\sum_{j\ge 3} \pj [\psi^\pm_j] 
= \pj [\zeta(\pm\mu\,\cdot)]+\sum_{{2^j\lesssim {\mu}}} \pj [\psi^\pm_j].
\end{align*}
 From Lemma \ref{local}, we have   $ \|\pj [\psi_j^\pm]\|_{1\to 2}\lesssim 2^{-j/2}{\mu}^{\frac{d-1}2}$  
   for $2^j\lesssim {\mu}$. 
Interpolation between this estimate and  \eqref{TT*} with $(1/p,1/q)=\mathfrak B$ and $(1,0)$, we obtain 
\[ \|\pj [\psi_j^\pm]\|_{p\to q} \lesssim 2^{j2d(\frac{2d-1}{4d}-\frac1q)}{\mu}^{ d(\frac1p+\frac1q)- \frac{2d+1}{2}}\]
for $2^j\lesssim {\mu}$ whenever $(1/p,1/q)$ is in the closed triangle $\mathcal T(d)$ with vertices $(1,1/2)$, $(1,0),$ and $\mathfrak B.$ Thus, choosing $q_0<\frac {4d}{2d-1}< q_1$ such that $(1/p, 1/q_0),$ $(1/p, 1/q_1) \in \mathcal T(d)$  and using $(a)$ in Lemma \ref{s-trick}, we obtain 
\begin{equation}\label{big}
 \| \sum_{2^j\lesssim {\mu}} \pj[\psi_j^\pm]\|_{L^p\to L^{q,\infty}} \lesssim  {\mu}^{ d(\frac1p+\frac1q)- \frac{2d+1}{2}}\end{equation}
for any $(1/p,1/q)\in (\mathfrak B,\mathfrak C].$ 

 We now handle $\pj [\zeta(\pm\mu\,\cdot)]$. 
Since $  \sum_{2^{-j}\lesssim {\mu}^{-1}}\psi(2^jt) = \zeta({\mu} t) $, by \eqref{TT*} and $(c)$ in Lemma \ref{s-trick}  we have the restricted-weak type $(p_\circ, q_\circ)$ estimate for $\pj [\zeta(\pm\mu\,\cdot)]$ with $(1/p_\circ,1/q_\circ)=\mathfrak B$: 
\[ \| \pj [\zeta(\pm\mu\,\cdot)] \|_{L^{p_\circ,1}\to L^{q_\circ,\infty}} \lesssim {\mu}^{-\frac12(\frac1{p_\circ}-\frac1{q_\circ})}={\mu}^{-\frac1{2d+1}}.\]
Interpolating this and the  estimates  $\|\pj [\zeta(\pm\mu\,\cdot)] \|_{1\to 2}\lesssim {\mu}^{\frac{d-2}2}$,  $\|\pj [\zeta(\pm\mu\,\cdot)]\|_{1\to\infty} \lesssim {\mu}^{d-1}$ which follow from Lemma \ref{local}, we obtain
\[ \|\pj [\zeta(\pm\mu\,\cdot)] \|_{p\to q}\lesssim  {\mu}^{ d(\frac1p-\frac1q)-1}\] 
whenever $(1/p,1/q)$ is in $\mathcal T(d)\setminus\{\mathfrak B\}.$ So, we get  
the estimate $
 \| \pj [\zeta(\pm\mu\,\cdot)]  \|_{p\to q}\lesssim  {\mu}^{ d(\frac1p+\frac1q)-\frac{2d+1}2}$
for $(1/p,1/q)\in(\mathfrak B,\mathfrak C]$ since $d(\frac1p-\frac1q)-1=d(\frac1p+\frac1q)-\frac{2d+1}2$ when $\frac1q=\frac{2d-1}{4d}$. Combining this with  \eqref{big}, we obtain \eqref{tem_goal} for $(1/p,1/q)\in(\mathfrak B,\mathfrak C].$

We now turn to $\sum_{k\ge5} \pj [\varphi_k]$ and $\pj[\varphi^0]$.  
Applying Lemma \ref{local}, we have  the estimate 
$\|\pj [\varphi_k]\|_{1\to 2}\lesssim 2^{-k/2}{\mu}^{\frac{d-1}2}$ for $2^k\lesssim {\mu}$ and  \eqref{TT*k}, especially, with $(1/p,1/q)=\mathfrak B, (1,0)$. We also have the restricted-weak type $(p,q)$ estimate for $\sum_{2^{-k}\lesssim {\mu}^{-1}}\pj[\varphi_k ]$ at $(p,q)=\mathfrak B$. 
Thus, in the same manner as before we can obtain 
\Be
\label{phi-k}  \|\sum_{k\ge5} \pj [\varphi_k] \|_{L^p\to L^{q,\infty}} \lesssim  {\mu}^{ d(\frac1p+\frac1q)-\frac{2d+1}2}
\Ee
for $(1/p,1/q)\in (\mathfrak B,\mathfrak C].$   Finally,  we have $\|\pj[\varphi^0]\|_{1\to \infty}\lesssim {\mu}^{-\frac12}$, $\|\pj[\varphi^0]\|_{p_\circ \to q_\circ}\lesssim {\mu}^{-\frac1{2d+1}}$ at $(1/{p_\circ},1/{q_\circ})=\mathfrak B$ from \eqref{TT*0}, and $\|\pj[\varphi^0]\|_{1\to 2}\lesssim {\mu}^{\frac{d-1}2}$ from Lemma \ref{local}. Thus, interpolation gives 
\Be
\label{phi-0}  \|\pj[\varphi^0]\|_{p\to q} \lesssim {\mu}^{d(\frac1p+\frac1q)-\frac{2d+1}2}
\Ee
for $(1/p,1/q)\in [\mathfrak B,\mathfrak C]$. 

Combining the estimates \eqref{tem_goal},  \eqref{phi-k}, and \eqref{phi-0}  together with  \eqref{decom1}, we get the desired weak type $(p,q)$  estimate for $\pj$ when $(1/p,1/q)\in (\mathfrak B,\mathfrak C]$.

\section{Proof of Theorem \ref{main} : Sharpness  }

In this section we show  the estimate \eqref{est-main}  is sharp and the failure of \eqref{est-main} for
$(1/p,1/q)\in [\mathfrak B, \mathfrak C]\cup[\mathfrak B', \mathfrak C']$. 

\begin{prop}\label{sharp} Let $d\ge1$ and $1\le p\le 2\le  q\le \infty$. For  ${\mu} $  large enough,  there is a constant $C,$ independent of ${\mu}$, such that 
\begin{align}
\| \pj \|_{p\to q} &\ge C {\mu}^{-\frac12(\frac1p-\frac1q)},\label{R1} \\
\| \pj \|_{p\to q} &\ge C {\mu}^{d(\frac1p-\frac1q)-1},\label{R3} \\
\| \pj \|_{p\to q} &\ge C {\mu}^{ \frac{2d-1}2-d(\frac1p+\frac1q)}. \label{R2}
\end{align}
\end{prop}

\begin{proof}[Proof of Theorem \ref{main}: Sharpness] By duality and \eqref{R2} we obtain
\[ \|\pj\|_{p\to q} \ge C{\mu}^{d(\frac1p+\frac1q)-\frac{2d+1}2}\]
for any $q\ge 2.$ Combining this and the estimates in Proposition \ref{sharp}, we obtain
\[ \|\pj\|_{p\to q}\ge C{\mu}^{\varrho(p,q)}\]
 for $1\le p\le 2\le q\le \infty$. 

The failure of the strong type estimate \eqref{est-main} for $p,q$ satisfying $(1/p,1/q)\in [\mathfrak B,\mathfrak  C]\cup[\mathfrak B', \mathfrak  C']$ can be shown by using the $L^p$--$L^q$ transplantation argument in \cite[Lemma 3.5]{JLR} (see the paragraph below Lemma 3.5 of \cite{JLR})  because  the twisted Laplacian $L$ is also an elliptic operator on $\R^{2d}$. To do show,  we define an projection operator $P$ by
\[P=\sum_{kn\le {\mu}\le (k+1)n}\pj \]
for large $k$, $n>0$, and set $P(z,z')$ the kernel of $P$. If we assume that $ \|\pj\|_{p\to q}\lesssim {\mu}^{d(\frac1p-\frac1q)-1} $, by the triangle inequality, we have
\[\|P\|_{p\to q} \lesssim k^{d(\frac1p-\frac1q)-1}n^{d(\frac1p-\frac1q)}.\]
This implies 
\begin{equation}\label{dual}
 n^d\Big| \iint P(z,z')f(n^{1/2}z')g(n^{1/2}z) dzdz' \Big| \lesssim  k^{d(\frac1p-\frac1q)-1}\|f\|_p\|g\|_{q'}
\end{equation}
for $f,g\in C^\infty_c(\R^{2d})$.
Let $f$ and $g$ be supported in a ball of radius $r$. If $z'$ and $z$ are  in the support of $f(n^{1/2}\cdot)$ and $g(n^{1/2}\cdot)$, respectively, $z-z'$ is in a ball of radius $2rn^{-1/2}$, hence $|z-z'|$ is  small enough if $n$ is sufficiently large. Applying H\"ormander's theorem \cite[Theorem 5.1]{H68} (see also Theorem 3.6 of \cite{JLR}) we see that 
\[ P(z,z') =(2\pi)^{-2d} \int_{kn\le|\xi|^2\le (k+1)n} e^{i\psi(z,z',\xi)} d\xi +\mathcal E(z,z',k,n)\]
where $\psi(z,z',\xi)=\langle (x-x',y-y'),\xi\rangle +O(|z-z'|^2|\xi|)$ with $z=x+iy, z'=x'+y'$, and $\mathcal E(z,z',k,n)=O(|kn|^{(2d-1)/2})$.  Thus, by rescaling $(z, z')\to (n^{-1/2}z,n^{-1/2}z')$  we see that 
the estimate \eqref{dual} implies 
\begin{align*}
 \Big| \iint \Big( \int_{k\le|\xi|^2\le (k+1)}& e^{i\psi(n^{-\frac12}(z, z'),n^{\frac12}\xi)} d\xi + O(k^{\frac{2d-1}2}n^{-\frac12})\Big)f(z')g(z) dzdz' \Big| \\
 &\lesssim  k^{d(\frac1p-\frac1q)-1}\|f\|_p\|g\|_{q'}.
\end{align*}
Letting $n$ to $\infty$, this yields
\begin{align*}
 \Big|\iint  \int_{k\le|\xi|^2\le (k+1)} e^{i \langle (x-x',y-y'),\xi\rangle } d\xi f(z')g(z) dzdz' \Big|\lesssim  k^{d(\frac1p-\frac1q)-1}\|f\|_p\|g\|_{q'}
\end{align*}
for any $f,g\in C^\infty_c(\R^{2d}),$ which is equivalent to 
\begin{align*}
  \Big\| \frac1{(2\pi)^{2d}} \int_{k\le|\xi|^2\le (k+1)} e^{i \langle (x,y),\xi\rangle } \widehat f(\xi) d\xi \Big\|_q\lesssim  k^{d(\frac1p-\frac1q)-1}\|f\|_p
\end{align*}
for any $f,g\in C^\infty_c(\R^{2d}).$ After scaling and letting $k\to \infty$, we obtain  the $2d$-dimensional restriction-extension estimate
\begin{equation}\label{re}
\Big\|\int_{\mathbb S^{2d-1} }\widehat f(\xi)e^{2\pi i (x,y)\cdot \xi}d\sigma(\xi)\Big\|_q\lesssim \|f\|_p. \end{equation}
It was already known that \eqref{re} is true only if $(1/p,1/q)\in {\mathcal R_3}$. {(See \cite{B86},\cite[Theorem 3.6]{JLR}).}
\end{proof}

\subsection*{Proof of the lower bounds \eqref{R1} and \eqref{R3}} We prove the estimates \eqref{R1} and \eqref{R3} by using duality argument and the known sharpness result obtained by Koch and Ricci \cite{KR}. In fact, we will use the fact that  there is a constant $C>0$, independent of ${\mu}$, such that for $2\le q\le \infty$
\begin{equation}\label{known_TT*} \|\pj\|_{q'\to q}\ge C{\mu}^{\varrho(q',q)},\end{equation}
which follows from  \eqref{known} and $TT^*$-argument.
Since $\varrho(p,q)=\max\{-\frac12(\frac1p-\frac1q),d(\frac1p-\frac1q)-1,  \frac{2d-1}2-d(\frac1p+\frac1q), d\big(\frac1p+\frac1q\big) -\frac{2d+1}{2}\},$ it is enough to show that \eqref{R1} on $(\frac1p,\frac1q)\in\mathcal R_1$ and \eqref{R3} on $(\frac1p,\frac1q)\in\mathcal R_3$.

We only show \eqref{R1} since the same argument works for \eqref{R2}. 
We prove  \eqref{R1} by contradiction. Suppose 
\eqref{R1} fails for some $p, q$ with $p\neq q'$,  then there are sequences $c_k$ and  $\mu_k$ such that 
\[ \|\mathcal P_{\mu_k} \|_{p\to q}= c_k{\mu}_k^{\varrho(p,q)},\]
$\mu_k\to \infty$, and $c_k\to 0$ as $k\to \infty$. Then,   by duality we also have $\|\mathcal P_{{\mu_k}} \|_{q'\to p'}= c_k{\mu}_k^{\varrho(p,q)}$. Interpolation between these two estimates we get 
\[ \|\mathcal P_{\mu_k} \|_{r\to s}\le c_k{\mu}_k^{\varrho(r,s)}  \] 
for all $r,s$ satisfying $(1/r,1/s)\in [(1/p,1/q), (1/q',1/p')]$ { with $1/r-1/s=1/p-1/q$.}  In particular we get 
$ \|\mathcal P_{\mu_k} \|_{r\to r' }\le c_k{\mu}_k^{\varrho(r,r')}$. 
This contradict \eqref{known_TT*} because $c_k\to 0$ as $k\to \infty$.

\subsection*{Proof of \eqref{R2}} We make use of the  formula  \eqref{pro_tw} 
where the twisted kernel $\varsigma_k$  is given by the Laguerre function. In fact,  let $\mathcal L^\alpha_k(t), t\ge0,$ denote the normalized Laguerre function of type $\alpha$  given by
\[\mathcal L^\alpha_k(t) = \Big(\frac{k!}{\Gamma(k+\alpha+1)}\Big)^{1/2} t^{\alpha/2} e^{-t/2}L^\alpha_k(t).\]
We clearly have  
\[\varsigma_k(z) =(\frac{k!}{(k+d-1)!})^{-1/2}|z|^{-({d-1})} \CL^{d-1}_k(|z|^2/2).\] 
There is a large body of literature concerning 
the asymptotic behavior of the Laguerre functions. 
We refer the reader to \cite{olv, FW, QW} and references therein. 
However,  for our purpose  we use the following relatively simple asymptotic formula.

\begin{lemma}\label{asymp}\cite[p.422]{Mu1}
Let $\alpha \ge 0$ and $k\in \N$. Then  
\[ \CL^\alpha_k(t) =\Big(\frac 2\pi\Big)^{1/2} \frac{(-1)^k}{t^{1/4}(\nu-t)^{1/4}}\cos\Big( \frac{ \nu(2\theta -\sin 2\theta)-\pi}{4} \Big)+O\Big(\frac{\nu^{1/4}}{(\nu-t)^{7/4}}+(\nu t)^{-3/4}\Big), 
\]
where $\nu= 4k+2\alpha+2$, $0<t<\nu$, and $\theta= \cos^{-1}(t^{1/2}\nu^{-1/2}).$
\end{lemma}

Recalling ${\mu}=2k+d,$ from   Lemma \ref{asymp} we have 
\begin{align}
\label{pj-asym} 
 \varsigma_k(z) = \Big(\frac 2\pi & \frac{(k+d-1)!}{k!}\Big)^{1/2}\frac{(-1)^k}{|z|^{{d-1}}}  \\ &\times \Big\{ ( |z|^2({\mu}-2^{-2}|z|^2))^{-1/4}\cos g(|z|)+O\big( \mu^{-3/2}\big)\Big\}
 \nonumber
 \end{align}
for $\sqrt\mu/8 \le |z|\le \sqrt\mu/2$, where 
\[ g(s) = \frac{\mu}2\Big(2\theta(s) -\sin 2\theta(s)\Big)-\frac\pi4,\quad \theta(s) =\cos^{-1}\Big(\frac{s}{2\sqrt\mu}\Big).\]

Note that $g$ is monotone decreasing and $\{g(t) : \sqrt\mu/8\le t\le \sqrt\mu/3 \}$ is an interval of length $\sim \mu$.  So,
 there exist  $\sqrt\mu/8<t_1<t_2<\cdots<t_N<\sqrt\mu/3$, $N\sim \mu$, such that $|\cos g(t_j)|=1$ for all $j$.\footnote{In fact, $g(t_j)=\pi(j+n_0)$ for some integer $n_0$.} Since $|g'(t)|\sim \sqrt{\mu}$, $t_{j+1}-t_j \sim 1/\sqrt\mu$ for all $j$. Also, $|\cos g(t)|\ge \cos(\pi/4)>0$ whenever $|t-t_j|\le \pi/(8\sqrt\mu)$.

To prove  \eqref{R2}, we set $D_j := [t_j, t_j +\pi/(8\sqrt\mu)]$, $1\le j\le N,$ and define $f$ on $\C^d$ by
\[ f(z):=\sum_{j=1}^N \chi_{D_j }(|z|) \varsigma_k(z).\]
Since $\supp f\subset B(0,\sqrt\mu/2)\setminus B(0,\sqrt\mu/8),$ it is easy to see that
\begin{align*}
\int_{\C^d}|f(z)|^p dz
&\lesssim \int_{\sqrt\mu/8}^{\sqrt{\mu}/2} {\mu}^{\frac{2d-3}4 p}\,r^{-(d-\frac12)p+2d-1} dr \lesssim {\mu}^{-p/2+d},
\end{align*}
where we use $-(d-\frac12)p+2d-1\ge0$, since $p\le 2$.

We now observe $\pj f$ near the origin. For $|z|\le \pi/(32\sqrt{\mu})$ and $w\in \C^d$ satisfying $|z-w|\in D_j,$  we have $ |w|\in [t_j -\pi/(8\sqrt\mu),t_j +5\pi/(32\sqrt\mu)]$ and $|\cos g(|w|)|\ge c>0$ for some $c>0$ independent of $\mu$.
This yields  $\varsigma_k(z-w)\varsigma_k(w)>0$ on $|z- w|\in D_j$ and $|z|\le \pi/(32\sqrt{\mu})$ if   $k$ is large enough. Thus, if ${\mu} $ is large enough,  for  $|z|\le \pi/(32\sqrt{\mu})$ we obtain  
\begin{align*}
|\pj f(z)|
&\ge (2\pi )^{-d} \sum_{j=1}^{N} \Big| \re \Big(\int_{\C^d} \chi_{D_j}(|z-w|) \varsigma_k(z-w) \varsigma_k(w) e^{i\frac12\im z\cdot \overline w} dw\Big)\Big|\\
&\gtrsim \sum_{j=1}^{N} {\mu}^{d-1-1/2} \int_{\C^d} \chi_{D_j}(|z-w|) |z-w|^{-(d-1)-1/2} |w|^{-(d-1)-1/2}  dw \\
&\gtrsim \sum_{j=1}^{N} {\mu}^{d-3/2} \int_{ t_j+\frac{\pi}{32\sqrt \mu}}^{ t_j+\frac{3\pi}{32\sqrt\mu}}  r^{-2(d-1)-1 +2d-1}  dr \sim {\mu}^{d-1}, 
\end{align*}
where the implicit constant is independent of ${\mu}.$ Indeed,  we used Stirling's formula to show $\frac{(k+d-1)!}{k!}\sim k^{d-1}\sim {\mu}^{d-1}$ in the second line, and $N \sim {\mu}$.
Therefore, for $1\le p\le 2$  and ${\mu}$ large enough, we get
\begin{align*}
\|\pj\|_{p\to q} &\ge  \|\pj f\|_q/\|f\|_p \ge  \|\pj f\|_{L^q(|z|\le \pi/(32\sqrt{{\mu}}))}/\|f\|_p\\
&\gtrsim {\mu}^{-d/q +d-1} {\mu}^{-d/p+1/2} \sim {\mu}^{\frac{2d-1}2 -(\frac dp+\frac dq)},
\end{align*}
which gives the lower bound \eqref{R2}.

\section{Resolvent estimate for the twisted Laplacian}

In this section, we prove the uniform resolvent estimates for $L$ in Theorem \ref{resol}.  
Here we closely follow the argument for the Hermite resolvent estimate in \cite{JLR} where the main ingredients were the uniform bound for Hermite spectral projection,  a kind of mixed norm estimate for the Hermite-Schr\"odinger propagator, and $L^p$--$L^q$ boundedness of the fractional Hermite operator.  For the twisted Laplacian $L$, the fractional integral operator  $L^{-s}$, $s>0,$ is  given by
\[ L^{-s} =\sum_{{\mu}} {\mu}^{-s} \pj =\frac1{\Gamma(s)} \int_0^\infty e^{-tL}t^{s-1} ds.\]
The $L^p$--$L^q$ boundedness was already established by Nowak and Stempak \cite{NS}. 
\begin{thm}\cite{NS} \label{fractional}
Let $s>0$ and $1\le p\le q\le \infty$. If $s>d$,  $L^{-s}$ is bounded from $L^p(\C^d)$ to $L^q(\C^d)$ for any 
$1\le p\le q\le \infty$, and $L^{-d}$ is bounded from $L^p(\C^d)$ to $L^q(\C^d)$ if and only if $(p,q)\neq (1,\infty)$.
In addition, if $s<d$, then $L^{-s}$ is bounded from $L^p(\C^d)$ to $L^q(\C^d)$ if and only if 
\[\frac1p-\frac sd\le \frac1q\quad  \text{and} \quad \Big(\frac1p,\frac1q\Big)\not=\Big(\frac sd,0\Big),\,\,\,\Big(1,\frac{d-s}d\Big).\]
\end{thm}

We also need the mixed norm estimate for the Schr\"odinger propagator $e^{-itL}$ in a certain range of $p,q$ as follows.

\begin{prop}\label{pro_strong} Let $d\ge2$ and $Q$ be the closed quadrangle with vertices 
$(1/2,1/2)$, $\mathfrak D'$, $\mathfrak F'$, $((d+1)/2d, (d-1)/2d))$ from which the two points 
$\mathfrak F'$ and $((d+1)/2d, (d-1)/2d))$  are removed. If $(1/p,1/q)\in Q$, then 
\begin{equation}\label{strong}
 \Big\| \int_{-\pi/2}^{\pi/2} |e^{-it L} f| dt\Big\|_{q}\lesssim \|f\|_p\end{equation}
 and we also have restricted-weak type estimates if $(1/p,1/q)= \mathfrak F'$ or $((d+1)/2d, (d-1)/2d))$. 
\end{prop}

From  \eqref{dispersive} it follows that 
\Be \label{dispersive2} \|e^{-itL} f\|_{\infty}\lesssim  |\sin t|^{-d} \|f\|_1.\Ee
Since $d\ge 2$, combining this with \eqref{iso},  the standard argument \cite{KT} yields  the end point Strichartz estimate
\Be
\label{stri}
\| e^{-itL} f \|_{{ L^2_t([-\frac\pi 2, \frac \pi2] ;L_z^{{2d}/({d-1})}(\mathbb C^d))}}\lesssim \|f\|_2.
\Ee

\begin{proof} From \eqref{iso} and \eqref{stri}
it is clear that \eqref{strong} holds for $(1/p,1/q)=(1/2,1/2)$,  $(1/p,1/q)=\mathfrak D'$.  Thus, in view of interpolation it suffices to show the restricted-weak type estimate 
\begin{equation}\label{strong_rw}
 \Big\| \int_{-\pi/2}^{\pi/2} |e^{-it L} f| dt\Big\|_{q,\infty}\lesssim \|f\|_{p,1}
\end{equation}
with  $(1/p,1/q)=\mathfrak F'$, $((d+1)/2d, (d-1)/2d))$.

To show \eqref{strong_rw},  we recall \eqref{partitionof1} and note that  from \eqref{iso}, \eqref{stri}, and \eqref{dispersive2}  the estimate 
\[\Big\|\int| {\psi^0}(t) e^{-itL}f |dt\Big\|_{q} \lesssim \|f\|_{p}\]  
holds with $(1/p,1/q)=(1/2,1/2), \mathfrak D', (1,0)$. By interpolation we see the above estimate holds for all $p,q$ satisfying   $(1/p,1/q)\in Q$. Thus, to show \eqref{strong_rw} we only have to show
\Be 
\label{r-w}
\Big\| \int |\sum_j\psi^\pm_j e^{-itL}f| dt \Big\|_{q,\infty}\lesssim  \|f\|_{p,1}
\Ee
with $(1/p,1/q)=\mathfrak F'$, $((d+1)/2d, (d-1)/2d))$. 
We now  claim that 
\Be 
\label{psij}
\Big\| \int |\psi^\pm_j e^{-itL}f| dt \Big\|_q\lesssim 2^{(\frac dp-\frac dq-1) j} \|f\|_p\Ee
holds provided that $(1/p,1/q)$ is contained in the closed triangle with vertices $(1/2,1/2), \mathfrak D',$ and 
$(1,0)$.  Once we have this, $(c)$ in Lemma \ref{s-trick}  gives the desired estimate \eqref{r-w} with $(1/p,1/q)=\mathfrak F'$, $((d+1)/2d, (d-1)/2d))$. See Figure \ref{fig1}.

It remains to show \eqref{psij}. From \eqref{stri}  the estimate $\|\psi_j^\pm(t) e^{-itL}\|_{L^2_tL_x^{\frac{2d}{d-1}}}\lesssim \|f\|_2$  follows.  Using this estimate, by H\"older's and Minkowski's inequalities we obtain 
$\|\int |\psi_j^\pm e^{-itL}f| dt\|_{\frac{2d}{d-1}} \lesssim 2^{-\frac {1}2 j}\|f\|_2$. We also have 
$\int |\psi_j^\pm e^{-itL}f| dt \lesssim 2^{ ({d-1}) j} \|f\|_1$ because of \eqref{represent} and  $\|\int |\psi_j^\pm e^{-itL}f| dt \|_2 \lesssim 2^{ - j} \|f\|_2$  from \eqref{iso}.
 Interpolation among these estimates gives  \eqref{psij} for $(1/p,1/q)$ in the closed triangle with vertices $(1/2,1/2),$ $\mathfrak D',$ and  $(1,0)$.
\end{proof}

Once we have the estimates in Proposition  \ref{pro_strong}, the desired resolvent estimates are established by following the argument used in \cite{JLR}. For completeness, however, we give a brief proof of Theorem \ref{resol}. We refer the reader to Section 8 of \cite{JLR} for the details.

\begin{proof}[Proof of Theorem \ref{resol}] The restricted-weak type estimates can be shown in the similar manner, so we only show 
the estimate \eqref{resol}. 
  Since the adjoint operator of $(L-z)^{-1}$ is $(L-\bar z)^{-1}$ which can be handled by the same argument, we may also assume  $1/p\le 1/q'$ and furthermore $(1/p,1/q)\in Q$. In fact, we show if the estimate \eqref{strong}  holds and  $L^{-1}$ is bounded from $L^p$ to $L^q$, then 
\eqref{resol} holds.

For simplicity we only consider the case $z\in \C$ with $\Re z>d-1/2$. The other cases can be handled by the same argument  which  we use to show the estimate for the term $\mathcal E$ below.
Since $\dist (z, 2\N+d)\ge c>0$, we write $z= 2n +d-2(a+ib)$ for some $n\in\N_0$, $a,b\in\R$ satisfying $|a|<1/2$ and $|(a,b)|\ge c/2.$  
Using a smooth symmetric function $\zeta$ supported in $(-1,1)$ and satisfying $\zeta(t)=1$ on $(-1/2,1/2)$, we decompose the resolvent operator $(L-z)^{-1}$ into two part; 
\[ (L-z)^{-1}f=\mathcal I f+\mathcal E f,\] where
\begin{align}\label{reso_decom}
\mathcal I f &:= \sum_{|k-n|< n} \frac{\zeta(\frac{k-n}n)}{2(k-n+(a+ib))} \mathcal P_{2k+d}f,
\\
\mathcal E f& := \sum_{ k} \frac{1-\zeta(\frac{k-n}n)}{2(k-n+(a+ib))} \mathcal P_{2k+d}f.
\end{align}
From the choice of $\zeta$, $\mathcal I$ is written as
\[\mathcal If =\mathcal I_1f+ \mathcal I_2f+ \mathcal I_3f, \]
where 
\begin{align*}
\mathcal I_1f
&:=
 \frac1{2(a+ib)}\mathcal P_{2n+d}f 
 \\
\mathcal I_2f&:= \sum_{k=1}^n \frac{(a+ib)\zeta(k/n)}{(k+a+ib)(-k+a+ib)}\mathcal P_{2(n-k)+d} f
\\
  \mathcal I_3f&:= \sum_{k=1}^n\frac{\zeta(k/n)}{2(k+a+ib)}\Big(\mathcal P_{2(k+n)+d}f-\mathcal P_{2(n-k)+d}f\Big).
\end{align*} 
Then,  for $p,q$ satisfying $(1/p,1/q)\in{ Q}$, we obtain 
\[ \|\mathcal I_1\|_{p\to q},\,\, \|\mathcal I_2\|_{p\to q}\lesssim 1\]  
uniformly in $n$ and $a,b$ satisfying $|(a,b)|\ge c/2$. Indeed, the estimate follows from the uniform bounds for $\pj$ which are direct consequence of Theorem \ref{main} (or Proposition \ref{pro_strong} with \eqref{represent0}). Using \eqref{represent0}, we see that 
\begin{align*}
  \mathcal I_3f
 &= \sum_{k=1}^n\frac{\zeta(k/n)}{2(k+a+ib)}\int_{-\pi/2}^{\pi/2}  (e^{2i tk}-e^{-2itk} ) e^{it(2n+d)}e^{-itL}f dt\\
 &= i \int_{-\pi/2}^{\pi/2}  \sum_{k=1}^n \frac{\zeta(k/n) \sin (2kt)}{k+a+ib}  e^{it(2n+d)}e^{-itL}f dt. \end{align*}
 Note that $|\sum_{k=1}^n \frac{\zeta(k/n) \sin (2kt)}{k+a+ib}|\le C$ uniformly in $n$ and $a,b$ obeying $|(a,b)|\ge c/2$. Combining this with Proposition \ref{pro_strong}, we get $\|\mathcal I_3\|_{p\to q}\lesssim 1$ uniformly in $n$ and $a,b$.
 
 The term $\mathcal E$ is easier to deal with. Since $\mathcal E f=m_{n}(L)\circ L^{-1} f$ with 
 \[ m_n(t)= t \Big(1-\zeta \Big(\frac{t-2n-d}{2n}\Big)\Big)/(2(t-z)),\quad z= 2n +d-2(a+ib), \]
 and $|\frac {d^l}{dt^l} m_n(t)|\lesssim (1+t)^{-l}$ for  $l=0,1,2,\cdots, d+2$ whenever $t>0$, applying  the Marcinkiewicz multiplier theorem \cite[Theorem 2.4.1]{Th93} and Theorem \ref{fractional}, we obtain the desired result.  
\end{proof}

\subsection*{Acknowledgements}
This work was  supported by the POSCO Science Fellowship and a KIAS Individual Grant no. MG070502 (E. Jeong) and Grant  no. NRF-2018R1A2B2006298 (S. Lee and J. Ryu).


\begin{thebibliography}{00}
\bibitem{bak} J.-G. Bak, \emph{Sharp estimates for the Bochner--Riesz operator of negative order in $\R^2$}, Proc. Amer. Math. Soc. {\bf 125} (1997), 1977--1986.

\bibitem{B86} L. B\"orjeson, {\it Estimates for the Bochner-Riesz operator with negative index}, Indiana U. Math. J. \textbf{35} (1986), 225--233. 

\bibitem{Car99} A. Carbery, A. Seeger, S. Wainger, J. Wright, {\it Class of singular integral operators along variable lines}, J. Geom. Anal. \textbf{9} (1999),  583--605.


\bibitem{Cu17} J.-C. Cuenin, \emph{Sharp spectral estimates for the perturbed Landau Hamiltonian with $L^p$ potentials,} Integral Equations Operator Theory \textbf{88} (2017), 127--141.

\bibitem{ev01} L. Escauriaza, L. Vega, \emph{Carleman inequalities and the heat operator II}, Indiana Univ. Math. J. \textbf{50} (2001), 1149--1169. 

\bibitem{FW} C. L.  Frenzen, R. Wong,  Uniform asymptotic expansions of Laguerre polynomials. SIAM J. Math. Anal. {\bf19} (1988), 1232--1248.


\bibitem{H68} L. H\"ormander, {\it The spectral function of an elliptic operator}, Acta Math. \textbf{121} (1968), 193--218.

\bibitem{JKL16}  E. Jeong, Y. Kwon, S. Lee,  \emph{Uniform Sobolev inequalities for second order non-elliptic differential operators,} Adv. Math. {\bf302} (2016), 323--350.

\bibitem{JLR}   E. Jeong, S. Lee, J. Ryu, \emph{Estimates for the Hermite spectral projection,} arXiv:2006.11762.

\bibitem{KT} M. Keel, T. Tao, \emph{Endpoint Strichartz estimates,} Amer. J. Math. \textbf{120} (1998), 955--980.

\bibitem{KRS87} C. E. Kenig, A. Ruiz, C. D. Sogge, \emph{Uniform Sobolev inequalities and unique continuation for second order constant coefficient differential operators}, Duke Math. J. \textbf{55} (1987),  329--347.

\bibitem{KST}  C. E. Kenig,  R. J.  Stanton,  P. A. Tomas,  \emph{Divergence of eigenfunction expansions,} J. Functional Analysis {\bf 46} (1982), 1, 28--44. 

\bibitem{KR} H. Koch, F. Ricci,  
\emph{Spectral projections for the twisted Laplacian,}
Studia Math. \textbf{180} (2007), no. 2, 103--110.

\bibitem{KT09}  H. Koch, D. Tataru, \emph{Carleman estimates and unique continuation for second order parabolic equations with nonsmooth coefficients},  Comm. Partial Differential Equations \textbf{34} (2009), 305--366. 

\bibitem{Mu1}  B. Muckenhoupt, \emph{Mean convergence of Hermite and Laguerre series. I,} Trans. Amer. Math. Soc. \textbf{147} (1970),  419--431. 


\bibitem{NS} A. Nowak, K. Stempak, \emph{ Potential operators and Laplace type multipliers associated with the twisted Laplacian,} Acta Math. Sci. Ser. B (Engl. Ed.) \textbf{37} (2017), no. 1, 280--292.

\bibitem{olv}  F. W. J. Olver, \emph{Asymptotics and Special Functions},  A K Peters/CRC Press; 2nd edition (1997).

\bibitem{QW}  W.-Y. Qiu,  R. Wong, \emph{Global asymptotic expansions of the Laguerre polynomials—a Riemann–Hilbert approach}, Numer Algor  {\bf 49} (2008), 331--372. 

\bibitem{RRT97} R. K. Ratnakumar, R. Rawat, S. Thangavelu, \emph{A restriction theorem for the Heisenberg motion group,} Studia Math. \textbf{126} (1997), no. 1, 1--12.

\bibitem{Stein} E. M. Stein, \emph{Harmonic Analysis: Real Variable Methods, Orthogonality and Oscillatory Integrals}, Princeton Univ. Press, Princeton, NJ, 1993.

\bibitem{SZ}  K. Stempak, J. Zienkiewicz,  \emph{Twisted convolution and Riesz means,}
J. Anal. Math. \textbf{76} (1998), 93--107.


\bibitem{Th91}  S. Thangavelu, \emph{Weyl multipliers, Bochner-Riesz means and special Hermite expansions,}  Ark. Mat. {\bf29} (1991), 307--321.


\bibitem{Th93} \bysame, \emph{Lectures on Hermite and Laguerre expansions}, Math. notes 42, Princeton University Press 1993.


\bibitem{Th98} \bysame, {\it Hermite and special Hermite expansions revisited}, Duke Math. J. \textbf{94} (1998), 257--278.

\bibitem{Th06} \bysame, \emph{Poisson transform for the Heisenberg group and eigenfunctions of the sublaplacian,} Math. Ann. \textbf{335} (2006), no. 4, 879--899.

\end{thebibliography}
\end{document}